\documentclass[10pt]{amsart}
\usepackage{amsmath}
\usepackage{amssymb} 

\def\beqnn{\begin{eqnarray*}}\def\eeqnn{\end{eqnarray*}}

\newtheorem{theorem}{Theorem}[section]
\newtheorem{lemma}[theorem]{Lemma}
\newtheorem{proposition}[theorem]{Proposition}

\theoremstyle{remark}
\newtheorem{remark}[theorem]{Remark}

\theoremstyle{definition}

\theoremstyle{problem}

\theoremstyle{conjecture}

\numberwithin{equation}{section}

\begin{document}

\begin{center}
\title{ On a multiplier operator induced by the Schwarzian derivative of univalent functions}
\end{center}

\author{Jianjun Jin}
\address{School of Mathematics Sciences, Hefei University of Technology, Xuancheng Campus, Xuancheng 242000, P.R.China}
\email{jin@hfut.edu.cn, jinjjhb@163.com}
\thanks{The author was supported by National Natural Science Foundation of China (Grant Nos. 11501157).}




\begin{abstract}
In this paper we study a multiplier operator which is induced by the Schwarzian derivative of univalent functions with a quasiconformal extension to the extended complex plane.   As  applications, we show that the Brennan conjecture is satisfied for a large class of quasidisks. We also establish a new characterization of asymptotically conformal curves and of the Weil-Petersson curves in terms of the multiplier operator. \end{abstract}

\maketitle
\pagestyle{plain}

\section{{\bf {Introduction}}}
We first fix some notations. Let $\Delta=\{z:|z|<1\}$ denote the unit disk in the
complex plane $\mathbb{C}$.  We denote the extended complex plane by $\widehat{\mathbb{C}}=\mathbb{C}\cup\{\infty\}$.
Let $\Delta^{*}=\widehat{\mathbb{C}}\setminus\overline{\Delta}$ be the
exterior of $\Delta$ and
$S^1=\partial\Delta=\partial\Delta^{*}$ be the unit circle. We use the notation $\Delta(z,r)$ to denote the disk centered at $z$ with radius $r$. We use $C(\cdot), C_{1}(\cdot), C_{2}(\cdot),\cdots$ to denote some positive numbers which depend only on the elements in the bracket.

Let $\mathcal{A}(\Delta)$ denote the class of all analytic functions in $\Delta$.   For $\alpha>1$, we define the Hilbert space $\mathcal{H}_{\alpha}(\Delta)$ as
$$\mathcal{H}_{\alpha}(\Delta)=\{\phi \in \mathcal{A}(\Delta) : \|\phi\|_{\alpha}^2:=(\alpha-1)\iint_{\Delta}|\phi (z)|^2(1-|z|^2)^{\alpha-2}dxdy<\infty\}.$$

Let $f$ be a univalent function in an open domain  $\Omega$ of $\mathbb{C}$, i.e., $f$ is a one to one analytic function in $\Omega$. The {\em Schwarzian derivative} $S_f$ of $f$ is defined as
$$S_f(z)=\bigg[\frac{f''(z)}{f'(z)}\bigg]'-\frac{1}{2}\left[\frac{f''(z)}{f'(z)}\right]^2, \,z\in \Omega.$$

Let $g$ be another univalent function in $f(\Omega)$. Then, we have
\begin{equation}\label{chain}S_{g \circ f}(z)=S_{g}(f(z))[f'(z)]^2+S_f(z),\,\, z\in \Omega.\end{equation}
For more properties on the Schwarzian derivative,  see \cite[Chapter II]{L}.

Let $f$ be a univalent function in $\Delta$. The {\em multiplier operator} $M_f$, induced by the Schwarzian derivative of $f$, is defined as
$$M_f(\phi)(z):=S_f(z)\phi(z), \,\,\phi\in \mathcal{A}(\Delta).$$

Let  $t\in \mathbb{R}$. We define the {\em integral means spectrum} $\beta_f(t)$ as the infimum of those numbers $\gamma>0$ such that there exists $C(f,\gamma)>0$ such that
$$I_t(f',r)=\int_{0}^{2\pi}|f'(re^{i\theta})|^td\theta \leq \frac{C(f,\gamma)}{(1-r)^{\gamma}},\,\, {\text{for}}\,\, r\in (0,1).$$
The {\em universal integral means spectrum} $B(t)$ is then the supremum of $\beta_f(t)$ taken over all univalent functions in $\Delta$. It is known (see \cite{SS}) that
\begin{equation}\label{ss}\beta_{f}(t)+1=\inf\{ \alpha>1, \,\,(f')^{\frac{t}{2}}\in \mathcal{H}_{\alpha}(\Delta)\}.\end{equation}

The famous Brennan conjecture states that $B(-2)=1$.  We say that the Brennan conjecture is satisfied for a simply connected plane domain $\Omega$ if $\beta_f(-2)\leq 1$ for any univalent function $f$ from $\Delta$ to $\Omega$. It is known that the Brennan conjecture is satisfied for some special types of domains, see \cite{BVZ}, \cite[page 286]{GM}.   In this paper, motivated by the work \cite{SS}, we study a multiplier operator which is induced by the Schwarzian derivative of univalent functions with a quasiconformal extension to $\widehat{\mathbb{C}}$, and show that the Brennan conjecture is satisfied for a large class of quasidisks. For more results on the Brennan conjecture and related topics,  see \cite{Carl, CM, M, HSo, HS-1}.

To present our results,  we first recall some basic definitions and properties of  quasiconformal mappings. We say a homeomorphism  $f$,  from one open domain $\Omega$ in
$\mathbb{C}$ to another,  is a quasiconformal mapping if it has locally
integral distributional derivatives and satisfies the Beltrami equation $\bar{\partial}f=\mu_f  \partial{f}$ with
$$\|\mu_f\|_{\infty}=\mathop{\text {ess sup}}\limits_{z\in \Omega} |\mu_f(z)|<1.$$
Here, the function $\mu_f$ is called the complex dilatation of $f$ and
$$\bar{\partial}f :=\frac{1}{2}\left(\frac{\partial}{\partial x}+i\frac{\partial}{\partial y}\right)f,\,\,\, \partial f :=\frac{1}{2}\left(\frac{\partial}{\partial x}-i\frac{\partial}{\partial y}\right)f.$$

Let $f$ be a quasiconformal mapping from one open domain $\Omega_1$ to another domain $\Omega_2$. If $g$ is another quasiconformal mapping from $\Omega_2$ to $\Omega_3$. Then the complex dilatations of $f$ and $g \circ f$ satisfy the following chain rule.
\begin{equation}\label{dila}
\mu_{g \circ f}(z) =\frac{\mu_f+(\mu_g\circ f(z))\kappa}{1+\overline{\mu_f}(\mu_g\circ f(z))\kappa},\,\, \kappa=\frac{\overline{\partial f}}{\partial f}.
\end{equation}

We say a Jordan curve $\Gamma$ in $\widehat{\mathbb{C}}$ is a quascircle if there is a quasiconformal mapping $f$ from $\widehat{\mathbb{C}}$ to itself such that $f(S^1)=\Gamma.$
The domain $f(\Delta)$ is then called a {\em quasidisk}. For more detailed introductions to the theory of quasiconformal mappings, see \cite{L} or \cite{LV}.

The following is the first result of this paper.
\begin{theorem}\label{m}
Let $\alpha>1$. Let $f$ be a univalent function in $\Delta$ admitting a quasiconformal extension to $\widehat{\mathbb{C}}$ with $\|\mu_f\|_{\infty}=k\in [0,1)$. Then, for any $\phi\in \mathcal{H}_{\alpha}(\Delta)$, we have
 \begin{eqnarray*}\|M_f(\phi)\|_{\alpha+4}^2=\|S_{f}(z)\phi(z)\|_{\alpha+4}^2
\leq \frac{36(\alpha+1)k^2}{(\alpha-1)}\|\phi (z)\|_{\alpha}^2.
 \end{eqnarray*}
\end{theorem}
\begin{remark}
Note that, for a univalent function $f$ in $\Delta$ admitting a quasiconformal extension to $\widehat{\mathbb{C}}$ with $\|\mu_f\|_{\infty}=k\in [0,1)$, we have
$$\sum_{m=1}^{\infty } m \left|\sum_{n=1}^{\infty} \gamma_{mn}\lambda_n\right|^2 \leq k^2\sum_{n=1}^{\infty}\frac{|\lambda_n|^2}{n}, \,\, \lambda_n\in \mathbb{C},$$
where $\gamma_{mn}$ are the Grunsky's coefficients of $f$, see \cite[Chapter 9]{Po}. Then Theorem \ref{m}  follows from the arguments in \cite{SS}.
\end{remark}

We use Theorem \ref{m} to show that
\begin{theorem}\label{m-1}
Let $f$ be a univalent function in $\Delta$ admitting a quasiconformal extension to $\widehat{\mathbb{C}}$.  If $\|\mu_f\|_{\infty} \leq \sqrt{\frac{5}{8}}\approx 0.79056$, then $\beta_{f}(-2) \leq 1$ and the Brennan conjecture is satisfied for the domain $f(\Delta)$.
\end{theorem}

\begin{remark}
Let $t\ \in \mathbb{R}$. In \cite{H}, Hedenmalm proved that, for a univalent function $f$ admitting a quasiconformal extension to $\widehat{\mathbb{C}}$ with  $\|\mu_f\|_{\infty}=k\in (0, 1)$, one has  
$$\beta_{f}(t) \leq \frac{1}{4} k^2 |t|^2(1+7k)^2,\,\, {\text {when}}\,\, |t| \leq \frac{2}{k(1+7k)^2},$$
and $$\beta_{f}(t) \leq k|t|-\frac{1}{(1+7k)^2},\,\, {\text {when}}\,\, |t| \geq \frac{2}{k(1+7k)^2}. $$

We consider $t=-2$. 

(1) When $k(1+7k)^2\leq 1$, i.e., $k\in (0, k_0]$, here $k_0\approx 0.18726$ is the real root of the equation $k(1+7k)^2=1$. We see that $
\beta_{f}(-2) \leq 1$ in this case.  

(2) When $k(1+7k)^2\geq 1$, i.e., $k\in [k_0, 1)$,  we have
$$\beta_{f}(-2) \leq 2k - \frac{1}{(1+7k)^2}=\frac{2k(1+7k)^2-1}{(1+7k)^2}.$$
Thus, if $k_0 \leq k<1$ and 
$$\frac{2k(1+7k)^2-1}{(1+7k)^2}\leq 1,$$
i.e., $k_0\leq k \leq k_1\approx0.52301$, here $k_1$ is the real root of the equation $$\frac{2k(1+7k)^2-1}{(1+7k)^2}=1,$$
then we have $\beta_{f}(-2) \leq 1$.

Consequently, we see that,  if $0\leq k\leq k_1\approx0.52301$,  then we have $\beta_{f}(-2) \leq 1$.  Hence Theorem \ref{m-1} provides an improvement of the results in \cite{H}. \end{remark}

The paper is organized as follows. We will give the proof of Theorem \ref{m-1} in the next section. By refining results in \cite{SS}, we will show in Section 3 that  the Brennan conjecture is  satisfied for another class of quasidisks. In Section 4,  we establish a new characterization of asymptotically conformal curves and of the Weil-Petersson curves in terms of the multiplier operator. We will present some remarks in Section 5.

\section{{\bf {Proof of Theorem \ref{m-1}}}}
We need the following result established by Shimorin in \cite{SS}.
\begin{proposition}\label{p-1}
If $f$ is a univalent function in $\Delta$ and 
\begin{eqnarray*}\|S_{f}(z)\phi(z)\|_{\alpha+4}^2
\leq \frac{36(\alpha+1)(\alpha+3)}{\alpha(\alpha+2)}\|\phi(z)\|_{\alpha}^2
 \end{eqnarray*}holds for any $\alpha>2$ and $\phi\in \mathcal{H}_{\alpha}(\Delta)$, then $\beta_{f}(-2)\leq 1$.
\end{proposition}
\begin{remark}
 Proposition \ref{p-1} is  Proposition 8 of \cite{SS}, where it is shown that under the assumption, $(f')^{-1}\in \mathcal{H}_{\alpha}(\Delta)$ for any $\alpha>2$.
\end{remark}
The following lemma  will be used later.
\begin{lemma}\label{le}
Let $f$ be a univalent function in $\Delta$ admitting a quasiconformal extension to $\widehat{\mathbb{C}}$.   If $g$ is another univalent function from $\Delta$ to $f(\Delta)$, then $g$ admits a quasiconformal extension to $\widehat{\mathbb{C}}$ with $\|\mu_g\|_{\infty}=\|\mu_f\|_{\infty}$.
\end{lemma}

\begin{proof}
We first notice that $g^{-1} \circ f|_{\Delta}=\sigma|_{\Delta}$, here $\sigma$  is a M\"obius transformation which maps the unit disk into itself. Since $\sigma \circ f^{-1}$ is a quasiconformal mapping from $\widehat{\mathbb{C}}$ to itself,  we see that  $\left[\sigma \circ f^{-1}\right]^{-1}=f \circ \sigma^{-1}$
is a quasiconformal  extension of  $g$ to $\widehat{\mathbb{C}}$.

Noting that $g^{-1}|_{ \widehat{\mathbb{C}}\setminus \overline{f{(\Delta)}}}=\sigma \circ f^{-1}|_{ \widehat{\mathbb{C}}\setminus \overline{f{(\Delta)}}}$, we see that $\mu_{g^{-1}}(z)=\mu_{f^{-1}}(z),\, z\in \mathbb{C}\setminus \overline{f{(\Delta)}}.$ On the other hand, we  have
$$|\mu_{f}(z)|=|\mu_{f^{-1}}(f(z))|\, \, {\text{and}}\,\, |\mu_{g}(z)|=|\mu_{g^{-1}}(g(z))|,\,\, z \in \mathbb{C}\setminus \overline{\Delta}.$$
It follows that $\|\mu_{g}\|_{\infty}=\|\mu_{f}\|_{\infty}.$ The lemma is proved.
\end{proof}

Now, we start to prove Theorem \ref{m-1}. By combining Theorem \ref{m}, Proposition \ref{p-1} and Lemma \ref{le}, we see that,  if \begin{equation}\label{inequ-1-0}\frac{36(\alpha+1)k^2}{(\alpha-1)} \leq \frac{36(\alpha+1)(\alpha+3)}{\alpha(\alpha+2)}\end{equation}
holds for any $\alpha>2$, then $\beta_f(-2)\leq 1$.

On the other hand, we see that the inequality (\ref{inequ-1-0}) is equivalent to
$$k^2\leq \frac{(\alpha-1)(\alpha+3)}{\alpha^2+2\alpha}=1-\frac{3}{\alpha^2+2\alpha}.$$
It is easy to see that
$$\inf\limits_{\alpha>2} \left[1-\frac{3}{\alpha^2+2\alpha}\right]=\frac{5}{8}.$$
Thus,  if $k^2\leq \frac{5}{8}$, i.e., $0\leq k\leq\sqrt{\frac{5}{8}}\approx 0.79056$, then the inequality (\ref{inequ-1-0}) holds for any $\alpha>2$.   This finishes the proof  of Theorem \ref{m-1}.

\section{\bf {Boundary dilatation of the quasiconformal extension of univalent functions and Brennan conjecture}}

\subsection{Statement of the results} In this section, by refining Proposition \ref{p-1}, we shall show that the Brennan conjecture is satisfied for another class of quasidisks.

To state the results of this section,  we introduce the notion of {\em boundary dilatation} of a quasiconformal mapping.  Let $f$ be a univalent function in $\Delta$ admitting a quasiconformal extension to $\widehat{\mathbb{C}}$.  The {\em boundary dilatation} of $f$, denoted by $h(f)$, is defined as
\begin{equation}\label{boun}
h(f):=\inf\{\|\mu_f|_{\Delta^{*}\setminus E}\|_{\infty}: \, E {\text { is a compact set in}}\,\, \Delta^{*}\}.
\end{equation}
Here we see $ \Delta^{*}$ as an open set in the Riemann sphere $\widehat{\mathbb{C}}$ under the spherical distance and $h(f)$ is the infimum of $\|\mu_f|_{\Delta^{*}\setminus E}\|_{\infty}$ over all compact subsets $E$ contained in $\Delta^{*}$. 

 We now state the main result of this section.
\begin{theorem}\label{m-1-1}
Let $f$ be a univalent function in $\Delta$ admitting a quasiconformal extension to $\widehat{\mathbb{C}}$.  Let $h(f)$ be the boundary dilatation of $f$. If $h(f)\leq \sqrt{\frac{3}{11}}\approx 0.52223$, then $\beta_f(-2)\leq 1$ and the Brennan conjecture is satisfied for the domain $f(\Delta)$.
\end{theorem}

For the proof of  Theorem \ref{m-1-1}, we need the following refinement of Proposition \ref{p-1}.
\begin{proposition}\label{m-1-1-1}
Let $f$ be a univalent function in $\Delta$, $r\in (0,1)$ and let $f_r(z)=f(rz)$. If, for any $\alpha>2$,  there is a constant $R=R(f, \alpha) \in (0,1)$ such that for any $r\in (0,1)$,
\begin{eqnarray*}\iint_{A_R}|S_{f_r}(z)\phi(z)|^2(1-|z|^2)^{\alpha+2}dxdy
\leq \frac{36(\alpha+1)}{\alpha(\alpha+2)}\|\phi (z)\|_{\alpha}^2
 \end{eqnarray*}holds for all $\phi \in \mathcal{H}_{\alpha}(\Delta)$, then $\beta_{f}(-2)\leq 1$.
 Here, $A_R:=\Delta\setminus \Delta(0,R)$ is an annulus.
\end{proposition}

\subsection{Proof of Proposition \ref{m-1-1-1}}
We will use the following lemma from \cite{SS}.
\begin{lemma}\label{le-1}
A function $\phi  \in \mathcal{H}_{\alpha}(\Delta)$ if and only if $\phi' \in  \mathcal{H}_{\alpha+2}(\Delta)$. Moreover, for any $\varepsilon$ such that $0<\varepsilon<\alpha(\alpha+1)$,
\begin{equation}\label{ieq-1}
\|\phi'\|_{\alpha+2}^2 \leq [\alpha(\alpha+1)+\varepsilon] \|\phi\|_{\alpha}^2+C_1(\phi, \varepsilon);
\end{equation}
\begin{equation}\label{ieq-2}
\|\phi\|_{\alpha}^2 \leq \frac{1}{[\alpha(\alpha+1)-\varepsilon] }\|\phi '\|_{\alpha+2}^2+C_2(\phi, \varepsilon),
\end{equation}
where the constants $C_1(\phi, \varepsilon)$ and $C_2(\phi, \varepsilon)$ depend only on finitely many first
Taylor coefficients of the function $\phi$.
\end{lemma}

We begin the proof of Proposition  \ref{m-1-1-1}. In view of  (\ref{ss}), we see that it is enough to show that $(f')^{-1}\in \mathcal{H}_{\alpha}(\Delta)$ for any fixed $\alpha>2$. First, we have
\begin{equation}\label{ind}
-\frac{d^3}{dz^3}\left[(f')^{-1}\right]=\frac{d}{dz}\left[S_f(z) (f')^{-1}\right]+S_f(z)\frac{d}{dz}\left[(f')^{-1}\right].
\end{equation}
Also, for $r\in (0, 1)$, $\phi \in \mathcal{A}(\Delta)$, it follows from the assumption that there is a constant $R>0$ such that
\begin{eqnarray} \iint_{A_R}|S_{f_r}(z)\phi (rz)|^2(1-|z|^2)^{\alpha+2}dxdy \leq \frac{36(\alpha+1)}{\alpha(\alpha+2)}\|\phi (rz)\|_{\alpha}^2.\nonumber
 \end{eqnarray}
Hence
\begin{eqnarray}\label{sch} \|S_{f_r}(z)\phi (rz)\|_{\alpha+4}^2&=&(\alpha+3)\iint_{\Delta}|S_{f_r}(z)\phi (rz)|^2(1-|z|^2)^{\alpha+2}dxdy\nonumber \\
& \leq&\frac{36(\alpha+1)(\alpha+3)}{\alpha(\alpha+2)} \|\phi (rz)\|_{\alpha}^2+C_3(\phi, \alpha, R).
 \end{eqnarray}
We let $$\mathbf{A}(\alpha):=\frac{36(\alpha+1)(\alpha+3)}{\alpha(\alpha+2)},$$
and we use $(f'_r)^{-1}$ to denote $[f'_r(z)]^{-1}$ for simplicity.
Note that $(f'_r)^{-1}=[rf'(rz)]^{-1}$.  Then, by using (\ref{ieq-2}) three times for the function $[f'(rz)]^{-1}$, we see from (\ref{ind}) that
\begin{eqnarray}\label{est-0}
\lefteqn{[\alpha(\alpha+1)(\alpha+2)(\alpha+3)(\alpha+4)(\alpha+5)-\varepsilon] \|(f'_r)^{-1}\|_{\alpha}^2} \nonumber \\
&&\leq \|\frac{d^3}{dz^3}\left[(f'_r)^{-1}\right]\|_{\alpha+6}^2+\frac{1}{r^2} C_4(f, \varepsilon)\nonumber \\
&&\leq  \left( \left\|\frac{d}{dz}\left[S_{f_r}(z)(f'_r)^{-1}\right]\right\|_{\alpha+6}+\left\|S_{f_r}(z)\frac{d}{dz}\left[(f'_r)^{-1}\right]\right\|_{\alpha+6}\right)^2+\frac{1}{r^2}C_4(f, \varepsilon).
\end{eqnarray}
Here $\varepsilon>0$ is a small number.

On the other hand, since $S_{f_r}(z)(f'_r)^{-1}=rS_f(rz)[f'(rz)]^{-1}$, we see from (\ref{ieq-1}) and (\ref{sch}) that
\begin{eqnarray}\label{est-1}
\lefteqn{\left\|\frac{d}{dz}\left[S_{f_r}(z)(f'_r)^{-1}\right]\right\|_{\alpha+6}}\nonumber \\
&&\leq \sqrt{(\alpha+4)(\alpha+5)+\varepsilon}\left\|S_{f_r}(z)(f'_r)^{-1}\right\|_{\alpha+4}+C_5(f, \varepsilon)\nonumber \\
&&\leq\sqrt{(\alpha+4)(\alpha+5)+\varepsilon}\cdot\sqrt{\mathbf{A}(\alpha)\|(f'_r)^{-1}\|_{\alpha}^2+\frac{1}{r^2}C_6(f, \alpha)}+C_5(f, \varepsilon)\nonumber\\
&&\leq  \sqrt{\mathbf{A}(\alpha)[(\alpha+4)(\alpha+5)+\varepsilon]}\|(f'_r)^{-1}\|_{\alpha}+\frac{1}{r}C_7(f, \alpha, \varepsilon).
\end{eqnarray}
 In the second inequality of (\ref{est-1}) we have used $[f'(rz)]^{-1}$ instead of $\phi(rz)$. Since $\left[(f'_r)^{-1}\right]'={f''(rz)}[f'(rz)]^{-2}, $ then, from (\ref{sch}) and  (\ref{ieq-1}) again,   we obtain
\begin{eqnarray}\label{est-2}
\lefteqn{\left\|S_{f_r}(z)\frac{d}{dz}\left[(f'_r)^{-1}\right]\right\|_{\alpha+6}\leq \sqrt{\mathbf{A}(\alpha+2)\left\| \frac{d}{dz}\left[(f'_r)^{-1}\right]\right\|_{\alpha+2}^2+C_8(f, \alpha)}} \nonumber \\
&&\leq \sqrt{\mathbf{A}(\alpha+2)\left\{[\alpha(\alpha+1)+\varepsilon]\|(f'_r)^{-1}\|_{\alpha}^2+\frac{1}{r^2}C_9(f, \varepsilon)\right\}+C_8(f, \alpha)} \nonumber \\
&& \leq  \sqrt{\mathbf{A}(\alpha+2)[\alpha(\alpha+1)+\varepsilon]} \|(f'_r)^{-1}\|_{\alpha}+\frac{1}{r}C_{10}(f, \alpha, \varepsilon).
\end{eqnarray}
In (\ref{est-1}) and (\ref{est-2}) we have used that $$\sqrt{A+B}\leq \sqrt{A}+\sqrt{B}, \,A>0,\, B>0.$$
Thus, combining (\ref{est-0}),  (\ref{est-1}),   (\ref{est-2}),  we obtain 
\begin{eqnarray*}
\lefteqn{[\alpha(\alpha+1)(\alpha+2)(\alpha+3)(\alpha+4)(\alpha+5)-\varepsilon] \|(f'_r)^{-1}\|_{\alpha}^2} \nonumber \\
&&\leq \left[\sqrt{\mathbf{A}(\alpha)[(\alpha+4)(\alpha+5)+\varepsilon]} +\sqrt{\mathbf{A}(\alpha+2)[\alpha(\alpha+1)+\varepsilon]}\right]^2\|(f'_r)^{-1}\|_{\alpha}^2\nonumber \\
&&\quad +\frac{1}{r}C_{11}(f, \alpha, \varepsilon)\|(f'_r)^{-1}\|_{\alpha}+\frac{1}{r^2}C_{12}(f, \alpha, \varepsilon).
\end{eqnarray*}
Let
$$\mathbf{B}(\alpha, \varepsilon):=\alpha(\alpha+1)(\alpha+2)(\alpha+3)(\alpha+4)(\alpha+5)-\varepsilon;$$
$$\mathbf{C}(\alpha, \varepsilon):= \left[\sqrt{\mathbf{A}(\alpha)[(\alpha+4)(\alpha+5)+\varepsilon]} +\sqrt{\mathbf{A}(\alpha+2)[\alpha(\alpha+1)+\varepsilon]}\right]^2.$$
It is not difficult to see that for $\alpha>2$ and $\varepsilon$ small enough, 
$$\mathbf{D}(\alpha, \varepsilon)=\mathbf{B}(\alpha, \varepsilon)-\mathbf{C}(\alpha, \varepsilon)>0.$$
Hence
\begin{eqnarray}\label{eee-1}
\mathbf{D}(\alpha, \varepsilon)\|(f'_r)^{-1}\|_{\alpha}^2\leq \frac{1}{r}C_{11}(f, \alpha, \varepsilon)\|(f'_r)^{-1}\|_{\alpha}+\frac{1}{r^2}C_{12}(f, \alpha, \varepsilon).
\end{eqnarray}
We conclude that there exist $\mathcal{R}\in (0,1), \mathcal{M}>0$,  such that $\|(f'_r)^{-1}\|_{\alpha}^2 \leq \mathcal{M}$ when $r>\mathcal{R}$.
Otherwise, there is an increasing sequence $\{r_n\}$ with $r_n<1$ and $r_n\rightarrow 1$ as $n \rightarrow \infty$, such that $\|(f'_{r_n})^{-1}\|_{\alpha}^2 \rightarrow \infty$. This contradicts the inequality (\ref{eee-1}).
On the other hand, by Fatou's lemma, we have
$\|(f')^{-1}\|_{\alpha}^2\leq  \varliminf_{r\rightarrow 1^{-}} \|(f'_r)^{-1}\|_{\alpha}^2.$
Consequently,  we see that $(f')^{-1}\in \mathcal{H}_{\alpha}(\Delta)$ for any fixed $\alpha>2$,  which finishes the proof.

\subsection{Proof of Theorem \ref{m-1-1}}
To prove Theorem \ref{m-1-1}, we will use the following key lemma.
\begin{lemma}\label{th-le}
Let $f$ be a univalent function in $\Delta$, which has a quasiconformal extension to $\widehat{\mathbb{C}}$. For $r\in (0,1)$, let $f_r(z)=f(rz)$ and $h(f)$ be the boundary dilatation of $f$. Then, for any $\varepsilon\in (0, 1-h(f))$ there is a constant $R \in (0,1)$ such that for any $r\in (0,1)$,
\begin{eqnarray*}\iint_{A_R}|S_{f_r}(z)\phi (z)|^2(1-|z|^2)^{\alpha+2}dxdy
\leq \frac{36(h(f)+\varepsilon)^2}{(\alpha-1)[1-(h(f)+\varepsilon)^2]}\|\phi (z)\|_{\alpha}^2
 \end{eqnarray*}
 holds for any $\alpha>1$ and $\phi \in \mathcal{H}_{\alpha}(\Delta)$.
 Here, $A_R=\Delta\setminus \Delta(0,R)$ is an annulus.
\end{lemma}

\begin{proof}[Proof of Lemma \ref{th-le}] We need an integral expression of the Schwarzian derivative of a univalent function which can be extended to a quasiconformal mapping in $\widehat{\mathbb{C}}$. This integral expression has appeared in \cite{AZ}.
For the completeness,  we will give a detailed derivation of this integral expression and clarify some arguments presented  in \cite{AZ}.

Let $\bar{f}=\tau \circ f$, here $\tau(z)=\frac{1}{f' (0)}[z-f(0)]$.
Then we have $\bar{f}(0)=0, {\bar{f}}'(0)=1$. We assume that $\bar{f}(z)$ have the
series expansion at origin as
$$\bar{f}(z)=z+a_2z^2+a_3z^3+\cdots.$$
The mapping 
$\widehat{f}=\varsigma \circ \bar{f} \circ \varsigma$,
$\varsigma(z)=\frac{1}{z}$ is 
univalent(conformal) in $ \Delta^{*}\setminus \{\infty\}$, and has the
series expansion at infinity
$$\widehat{f}(z)=z+b_0+\frac{b_1}{z}+\cdots.$$
It is  easy to see that $b_0=-a_2, \, b_1=a_2^2-a_3$. For any $z\in \Delta^{*}\setminus \{\infty\}$, let
$$\phi_z(w)=\frac{w+z }{1+{\bar{z}}w}.$$
The Koebe transformation $\mathcal{K}_{\widehat{f}}(w)$ of $\widehat{f}$ (see  \cite[page 21]{Po}) is defined as
$$\mathcal{K}_{\widehat{f}}(w)=\frac{\widehat{f}(\phi_z(w))-\widehat{f}(z)}{(1-|z|^2){\widehat{f}}'(z)}.$$
It follows that
$F(w)=\varsigma \circ \mathcal{K}_{\widehat{f}} \circ \varsigma$ is univalent in $\Delta^{*}\setminus \{\infty\}$ and has a series expansion
at infinity
$$F(w)=w+c_0+\frac{c_1}{w}+\cdots.$$
Then, by Pompieu's formula,  for any $z\in \Delta^{*}\setminus \{\infty\}$, we have
$$F(z)=\frac{1}{2\pi i}\oint_{\Gamma} \frac{f(w)}{w-z}dw-\frac{1}{\pi}\iint_{\Delta(z, |z|+2)}\frac{\bar{\partial} F(w)}{w-z}dudv.$$
Here,  $\Gamma=\partial{\Delta}(z, |z|+2)$ is a circle. Since $F$ is univalent in $\Delta^{*}\setminus \{\infty\}$, it follows that $\bar{\partial}F(w)=0$ when $w \in \Delta^{*}\setminus \{\infty\}$. Then, by Laurent's theorem,  we have
$$F(z)=z+c_0-\frac{1}{\pi}\iint_{\Delta}\frac{\bar{\partial} F(w)}{w-z}dudv,\,\,z\in \Delta^{*}\setminus \{\infty\}.$$
Consequently,
\begin{eqnarray}\label{az-0}
\iint_{\Delta}\bar{\partial}F(w)dudv&=&\lim_{z\rightarrow
\infty}{z}^2\iint_{\Delta}\frac{\bar{\partial}F(w)}{(z-w)^2}
\,dudv  \\
&=& -\pi \lim_{z \rightarrow \infty }{z}^2
(F'(z)-1) \nonumber \\
&=& \pi c_1=-\frac{\pi}{6}\lim_{z\rightarrow \infty}{z}^4
S_{F}(z). \nonumber \end{eqnarray}
Note that $F= \varsigma\circ \chi \circ \widehat{f}\circ \phi_z\circ \varsigma$, where $$ \chi (w)=\frac{w-\widehat{f}(z)}{(1-|z|^2)\widehat{f}'(z)}.$$ Let
$$\rho_{z}(w):=\phi_{z}\circ \varsigma(w)=\frac{1+wz}{w+\bar{z}}.$$ From the transformation rule (\ref{chain}), it follows  that
$$S_{F}(w)=S_{\widehat{f}}\circ
\rho_z(w)[\rho_z'(w)]^2.$$
Consequently, we see from $$\rho_z'(w)=\frac{|z|^2-1}{(w+\bar{z})^2},$$ and $\rho_z(w) \rightarrow z$ as $w \rightarrow \infty$ that
$$S_{\widehat{f}}(z)=\lim_{w \rightarrow \infty}S_{F}(w)[\rho_z'(w)]^{-2}=(|z|^2-1)^{-2}\lim_{w\rightarrow \infty}{w}^4 S_{F}(w),$$
and then, in view of (\ref{az-0}), we obtain \beqnn
S_{\widehat{f}}(z)&=&-\frac{6}{\pi}(|z|^2-1)^{-2}\iint_{\Delta}\bar{\partial}F(w)\,dudv \\
&=& -\frac{6}{\pi}(|z|^2-1)^{-2}\iint_{\Delta}\mu_{F}(w)\partial F(w)\,dudv.\eeqnn
It follows that
 \beqnn
|S_{\widehat{f}}(z)|&\leq &\frac{6}{\pi}(|z|^2-1)^{-2}\iint_{\Delta}|\mu_{F}(w)||\partial F(w)|\,dudv
\\ &=&\frac{6}{\pi}(|z|^2-1)^{-2}\iint_{\Delta}|\mu_{F}(w)|\bigg[\frac{J_{F}(w)}{1-|\mu_{F}(w)|^2}\bigg]^{\frac{1}{2}}\,dudv,\eeqnn
where $J_F$ is the Jacobian of $F$. Hence, by Cauchy's inequality, we have \begin{equation}\label{inequ} |S_{\widehat{f}}(z)|^2 \leq \frac{36}{\pi^2} (|z|^2-1)^{-4}\iint_{\Delta}\frac{|\mu_F(w)|^2}{1-|\mu_{F}(w)|^2}\,dudv \iint_{\Delta}J_F(w)\,dudv.\end{equation}
By the well-known area theorem, we have $$
\iint_{\Delta}J_F(w)\,dudv\leq \pi. $$

On the other hand, since $\mu_F(w)=\mu_{\widehat{f} \circ \rho_z}(w)$, a change of variables in the first integral of inequality (\ref{inequ}) gives
\begin{equation*}|S_{\widehat{f}}(z)|^2(|z|^2-1)^2 \leq \frac{36}{\pi}
\iint_{\Delta}\frac{|\mu_{\widehat{f}}(\zeta)|^2}{1-|\mu_{\widehat{f}}(\zeta)|^2}\cdot \frac{d\xi d\eta}{|\zeta-z|^{4}}, \, z\in \Delta^{*}\setminus \{\infty\}.\end{equation*}
From (\ref{chain}), we have $$|S_{\widehat{f}}(z)|=|S_{f}(\frac{1}{z})|\frac{1}{|z|^4},
 z\in \Delta^{*}\setminus\{\infty\}. $$
 Thus, for any $z\in \Delta\setminus \{0\}$, we have
 \begin{equation}\label{add}
 |S_f(z)|^2(1-|z|^2)^2 \leq  \frac{36}{\pi}\iint_{\Delta}\frac{|\mu_{\widehat{f}}(\zeta)|^2}{1-|\mu_{\widehat{f}}(\zeta)|^2}\cdot \frac{d\xi d\eta}{|1-\zeta z|^{4}}.
 \end{equation}
 It is easy to see that (\ref{add}) still holds for $z=0$. Hence (\ref{add}) holds for all $z\in \Delta$.
For any $r\in (0, 1)$, we see from $S_{f_r}(z)=r^2S_f(rz)$ that
\begin{eqnarray}\label{add-1}
\lefteqn{|S_{f_r}(z)|^2(1-|z|^2)^4} \\
&&\leq \frac{36r^4}{\pi} (1-|z|^2)^2\iint_{\Delta}\frac{|\mu_{\widehat{f}}(\zeta)|^2}{1-|\mu_{\widehat{f}}(\zeta)|^2}\cdot \frac{d\xi d\eta}{|1-r\zeta z|^{4}}, \, z\in \Delta.\nonumber
\end{eqnarray}
Since $|\mu_{f}(z^{-1})|=|\mu_{\widehat{f}}(z)|$, $z\in \Delta\setminus \{0\}$, we see that for any $\varepsilon \in (0, 1-h(f))$ there is a constant $r_0 \in (0,1)$ such that
 \begin{equation*}
 |\mu_{\widehat{f}}(\zeta)|\leq  h(f)+\frac{\varepsilon}{2}
 \end{equation*}
for all $r_0\leq |\zeta| <1$.
Let
 $$\mathbf{E}(r, z):=\iint_{\Delta}\frac{|\mu_{\widehat{f}}(\zeta)|^2}{1-|\mu_{\widehat{f}}(\zeta)|^2}\cdot \frac{d\xi d\eta}{|1-r\zeta z|^{4}}.$$
From
$$|1-r\zeta z|^4 \geq (1-r_0)^4, \,\,{\text  {for any}}\,\,|\zeta|\leq r_0, z\in \Delta,$$
we see that
\begin{eqnarray*}\lefteqn{\mathbf{E}(r, z)\leq\frac{[h(f)+\frac{1}{2}\varepsilon]^2}{1-[h(f)+\frac{1}{2}\varepsilon]^2}\iint_{A_{r_0}} \frac{d\xi d\eta}{|1-r\zeta z|^{4}}}\\ &&\quad\quad\quad  +\frac{\|\mu_{\widehat{f}}\|_{\infty}^2}{1-\|\mu_{\widehat{f}}\|_{\infty}^2}\cdot\iint_{\Delta(0, r_0)} \frac{d\xi d\eta}{|1-r\zeta z|^{4}}\\
 &&\quad\leq \frac{[h(f)+\frac{1}{2}\varepsilon]^2}{1-[h(f)+\frac{1}{2}\varepsilon]^2}\iint_{\Delta} \frac{d\xi d\eta}{|1-r\zeta z|^{4}}+\frac{\|\mu_{\widehat{f}}\|_{\infty}^2}{1-\|\mu_{\widehat{f}}\|_{\infty}^2}\cdot  \frac{\pi r_0^2}{(1-r_0)^{4}}\\
 && \quad =\frac{[h(f)+\frac{1}{2}\varepsilon]^2}{1-[h(f)+\frac{1}{2}\varepsilon]^2} \cdot  \frac{\pi}{(1-|rz|^{2})^2}+\frac{\|\mu_{\widehat{f}}\|_{\infty}^2}{1-\|\mu_{\widehat{f}}\|_{\infty}^2}\cdot  \frac{\pi r_0^2}{(1-r_0)^{4}},
 \end{eqnarray*}
 where $A_{r_0}=\Delta \setminus \Delta(0, r_0)$.
It follows that
 \begin{eqnarray}\label{new-1}
\lefteqn{\frac{36r^4}{\pi} (1-|z|^2)^2\iint_{\Delta}\frac{|\mu_{\widehat{f}}(\zeta)|^2}{1-|\mu_{\widehat{f}}(\zeta)|^2}\cdot \frac{d\xi d\eta}{|1-r\zeta z|^{4}}} \nonumber \\
 && \leq \frac{36r^4[h(f)+\frac{1}{2}\varepsilon]^2}{1-[h(f)+\frac{1}{2}\varepsilon]^2} \cdot  \frac{(1-|z|^2)^2}{(1-|rz|^{2})^2}+\frac{\|\mu_{\widehat{f}}\|_{\infty}^2}{1-\|\mu_{\widehat{f}}\|_{\infty}^2}\cdot  \frac{36r^{4}r_0^2(1-|z|^2)^2}{(1-r_0)^{4}} \nonumber\\ && \leq
 \frac{36[h(f)+\frac{1}{2}\varepsilon]^2}{1-[h(f)+\frac{1}{2}\varepsilon]^2}+\frac{\|\mu_{\widehat{f}}\|_{\infty}^2}{1-\|\mu_{\widehat{f}}\|_{\infty}^2}\cdot  \frac{36(1-|z|^2)^2}{(1-r_0)^{4}}.
\end{eqnarray}
We see from (\ref{new-1}) that there is a constant $R\in (0,1)$ such that
 \begin{eqnarray}\label{add-2}
\frac{36r^4}{\pi} (1-|z|^2)^2\iint_{\Delta}\frac{|\mu_{\widehat{f}}(\zeta)|^2}{1-|\mu_{\widehat{f}}(\zeta)|^2}\cdot \frac{d\xi d\eta}{|1-r\zeta z|^{4}}
 \, \leq
 \frac{36(h(f)+\varepsilon)^2}{1-(h(f)+\varepsilon)^2}
\end{eqnarray}
holds for $R\leq |z|<1$. Here we have used that the function $x^2(1-x^2)^{-1}$ is increasing in $[0,1)$ and $h(f)+\varepsilon<1$.
Consequently, from (\ref{add-1}) and (\ref{add-2}), we find that
\begin{eqnarray*}\iint_{A_R}|S_{f_r}(z)\phi(z)|^2(1-|z|^2)^{\alpha+2}dxdy
\leq \frac{36(h(f)+\varepsilon)^2}{(\alpha-1)[1-(h(f)+\varepsilon)^2]}\|\phi(z)\|_{\alpha}^2
 \end{eqnarray*}
holds for any $r \in (0,1)$ and $\phi \in \mathcal{H}_{\alpha}(\Delta)$. The lemma is proved.
 \end{proof}

The following variant of Lemma \ref{le} can be established in analogous form, and will be stated without proof.
\begin{lemma}\label{le-2}
Let $f$ be a univalent function in $\Delta$ admitting a quasiconformal extension to $\widehat{\mathbb{C}}$.   If $g$ is another univalent function from $\Delta$ to $f(\Delta)$, then $g$ admits a quasiconformal extension to $\widehat{\mathbb{C}}$ with $h(g)=h(f)$.
\end{lemma}

We can now proceed with the proof of  Theorem \ref{m-1-1}.  We see from Proposition \ref{m-1-1-1}, Lemma \ref{th-le} and Lemma \ref{le-2} that,  if the inequality \begin{equation}\label{inequ-1}\frac{36h^2(f)}{(\alpha-1)(1-h^2(f))} \leq \frac{36(\alpha+1)}{\alpha(\alpha+2)}\end{equation}
holds for any $\alpha>2$, then $\beta_f(-2)\leq 1$. 

Meanwhile, it follows from (\ref{inequ-1}) that
$$\frac{h^2(f)}{1-h^2(f)} \leq \frac{\alpha^2-1}{\alpha^2+2\alpha}=1-\frac{2\alpha+1}{\alpha^2+2\alpha},\,\,\, {\text {and}} \,\,\,\,\inf\limits_{\alpha>2} \left[1-\frac{2\alpha+1}{\alpha^2+2\alpha}\right]=\frac{3}{8}.$$
Consequently,  if $$\frac{h^2(f)}{1-h^2(f)} \leq \frac{3}{8}, {\text { i.e.,}}\, \,0\leq h(f) \leq\sqrt{\frac{3}{11}}\approx 0.52223,$$ then the inequality (\ref{inequ-1}) holds for any $\alpha>2$ and $\beta_f(-2)\leq 1$. This finishes the proof. 

\section{ \bf {A new characterization of  asymptotically conformal curves and of the Weil-Petersson curves}}
Let $f$ be a univalent function from $\Delta$ to a bounded Jordan domain in $\mathbb{C}$.  We say that $f(S^1)$ is an {\em asymptotically conformal curve} if $f$ can be extended to a quasiconformal mapping in $\widehat{\mathbb{C}}$ and whose complex dilatation $\mu_f$ satisfies that $\mu_f(z) \rightarrow 0, \,\,|z|\rightarrow 1^{+}.$  Here, we say $f$ is an {\em asymptotically conformal mapping}.  See \cite{GS, Po-1, Sh-1}.
\begin{remark}
We see that $h(f)=0$ if $f$ is an asymptotically conformal mapping. Then, from Theorem \ref{m-1-1}, we obtain that the Brennan conjecture is satisfied for the quasidisk $f(\Delta)$ when $f$ is an asymptotically conformal mapping.
\end{remark}

We say that $f(S^1)$ is a {\em Weil-Petersson curve}, if $f$ can be extended to a quasiconformal mapping in $\widehat{\mathbb{C}}$ and whose complex dilatation $\mu_f$ satisfies that $$\iint_{{\Delta}^{*}}|\mu_f(z)|^2/(|z|-1)^2dxdy<\infty.$$  The Weil-Petersson curves have been studied extensively, see \cite{C, B-1, B-2, HS, Sh-1, Sh-2, W}.
It is known that  $f(S^1)$ is an asymptotically conformal curve if and only if
$S_f(z)(1-|z|^2)^2\rightarrow 0,\,\, |z| \rightarrow 1^{-}.$
Moreover,  $f(S^1)$ is a Weil-Petersson curve if and only if $$\iint_{\Delta}|S_{f}(z)|^2(1-|z|^2)^2dxdy<\infty.$$

We shall prove that
\begin{theorem}\label{asy}
Let $\alpha>1$.  Let $f$ be a univalent function from $\Delta$ to a bounded Jordan domain in $\mathbb{C}$.

(I) $f(S^1)$ is an asymptotically conformal curve if and only if the multiplier operator $M_f$,  acting from $\mathcal{H}_{\alpha}(\Delta)$ to $\mathcal{H}_{\alpha+4}(\Delta)$,  is compact.  Moreover,

(II) $f(S^1)$ is a Weil-Petersson curve if and only if the multiplier operator $M_f$ belongs to the Hilbert-Schmidt class.
\end{theorem}
{\em Proof of the sufficiency of (I) of Theorem \ref{asy}. }Suppose that $f(S^1)$ is an asymptotically conformal curve.  To show that $M_f$ is compact operator, it is sufficient to show that $M_f(\psi_n) \rightarrow 0$ for each sequence $(\psi_n)$ which converges to zero weakly. It is easy to check that $(\psi_n)$ converges to zero weakly if and only if $(\psi_n)$ is bounded and $(\psi_n)$ converges to zero locally.

On the other hand, we recall that $f(S^1)$ is an asymptotically conformal curve if and only if $S_f(z)(1-|z|^2)^2\rightarrow 0,\,\, |z| \rightarrow 1^{-}.$ Thus, for any $\varepsilon>0$, there exists  some $r\in (0,1)$ such that  $|S_f(z)(1-|z|^2)^2|<\varepsilon$, when $|z|>r$. It follows that, for $\psi \in \mathcal{H}_{\alpha}(\Delta)$,  we have 
\begin{eqnarray}\|M_f(\psi)\|_{\alpha+4}^2&=&(\alpha+3)\iint_{\Delta}|S_{f}(z)\psi(z)|^2 (1-|z|^2)^{\alpha+2}dxdy  \nonumber \\
&=&(\alpha+3)\iint_{\Delta}|S_{f}(z)|^2(1-|z|^2)^4|\psi(z)|^2 (1-|z|^2)^{\alpha-2}dxdy   \nonumber \\
&\leq & 36(\alpha+3)\iint_{|z|<r}|\psi(z)|^2 (1-|z|^2)^{\alpha-2}dxdy+ (\alpha+3){\varepsilon}^2\|\psi\|_{\alpha}^2.  \nonumber
\end{eqnarray}
Consequently,  we see that $M_f(\psi_n) \rightarrow 0$ for each sequence $(\psi_n)$ which converges to zero weakly. The sufficiency of (I) is proved. 

{\em Proof of the necessity of (I) of Theorem \ref{asy}.} If $M_f$ is a compact, we consider the function
$$\psi_a(z)=\frac{(1-|a|^2)^{\frac{\alpha}{2}}}{(1-az)^{\alpha}},\, a\in \Delta.$$
From \cite[Lemma 4.2.2]{Zh}, we see that $\psi_a(z) \in \mathcal{H}_\alpha(\Delta)$ and that $\psi_a(z)$ tends to zero locally uniformly in $\Delta$ when $|a|\rightarrow 1^{-}$. We conclude that $\psi_a$ converges to zero weakly, hence $M_f(\psi_a) \rightarrow 0$  as $|a| \rightarrow 1^{-}$, i.e., 
\begin{equation}\label{as-1}\lim\limits_{|a|\rightarrow 1^{-}}\iint_{\Delta} |S_f(z)|^2\frac{(1-|a|^2)^{\alpha}(1-|z|^2)^{\alpha+2}}{|1-az|^{2\alpha}}dxdy=0.\end{equation}
For $a\in \Delta$, let $l \in (0,1)$ be such that the disk $\Delta(a, l(1-|a|))=\{|z-a|\leq l(1-|a|)\}$ is contained in $\Delta$. Hence, for any $z\in \Delta(a, l(1-|a|))$, 
\begin{equation}\label{eq-1}(1-l)(1-|a|)\leq 1-|z| \leq (1+l)(1-|a|)\end{equation}
and
\begin{equation}\label{eq-1-1}  (1-|a|)\leq |1-az|\leq (2+l)(1-|a|) .\end{equation}
It follows from (\ref{eq-1}) and (\ref{eq-1-1}) that
\begin{equation}\label{as-2} \frac{(1-|a|^2)^{\alpha}(1-|z|^2)^{\alpha}}{|1-az|^{2\alpha}} \geq \frac{(1-l)^{\alpha}}{(2+l)^{2\alpha}}\end{equation}
holds for any $z\in \Delta(a, l(1-|a|))$.

On the other hand, since $|S_f(z)|^2$ is a subharmonic function 
$$|S_{f}(a)|^2(1-|a|^2)^2\leq \frac{4}{\pi l^2}\iint_{|z-a|<l(1-|a|)}|S_f(z)|^2dxdy.$$
It follows from (\ref{eq-1}) that
\begin{equation}\label{as-3}|S_{f}(a)|^2(1-|a|^2)^4\leq \frac{16}{\pi l^2(1-l)^2}\iint_{|z-a|<l(1-|a|)}|S_f(z)|^2(1-|z|^2)^2dxdy.\end{equation}
Combining (\ref{as-2}),  (\ref{as-3}),  we see that there is a constant $C(l, \alpha)>0$ such that
\begin{equation}|S_{f}(a)|^2(1-|a|^2)^4\leq C(l, \alpha) \iint_{\Delta} |S_f(z)|^2\frac{(1-|a|^2)^{\alpha}(1-|z|^2)^{\alpha+2}}{|1-az|^{2\alpha}}dxdy.\end{equation}
Thus, from (\ref{as-1}),  $S_f(a)(1-|a|^2)^2\rightarrow 0,\,\, |a| \rightarrow 1^{-}.$ Hence $f(S^1)$ is an asymptotically conformal curve. This finishes the proof of  (I). 

For part (II), let $n\in \mathbb{N}\cup\{0\}$ and let
$$e_{n}(z)=\sqrt{\frac{\Gamma(n+\alpha)}{n!\Gamma(\alpha)}}z^n,\, z\in \Delta. $$
Here, $\Gamma(s)$  stands for the usual Gamma function.  It is easy to see that $\{e_n\}$ is an orthonormal set in $\mathcal{H}_{\alpha}(\Delta)$.

It is  known that $M_f$ belongs to the Hilbert-Schmidt class if and only if
$$\sum_{n=0}^{\infty} \|M_f(e_n)\|_{\alpha+4}^2<\infty.$$
Since
$$\frac{1}{(1-|z|^2)^{\tau}}=\sum_{n=0}^{\infty} \frac{\Gamma(\tau+n)}{n! \Gamma(\tau)}|z|^{2n},\,\, \tau>0,\,z \in \Delta,$$
we have
\begin{eqnarray}\lefteqn{\sum_{n=0}^{\infty} \|M_f(e_n)\|_{\alpha+4}^2}\nonumber \\
&&=(\alpha+3)\sum_{n=0}^{\infty}\iint_{\Delta}|S_f(z)|^2 \frac{\Gamma(n+\alpha)}{n!\Gamma(\alpha)}|z|^{2n}(1-|z|^{2})^{\alpha+2}dxdy\nonumber \\
&&=(\alpha+3)\iint_{\Delta}|S_f(z)|^2 \sum_{n=0}^{\infty}\frac{\Gamma(n+\alpha)}{n!\Gamma(\alpha)}|z|^{2n}(1-|z|^{2})^{\alpha+2}dxdy\nonumber \\
&&=(\alpha+3)\iint_{\Delta}|S_f(z)|^2 (1-|z|^{2})^{2}dxdy,\nonumber
\end{eqnarray}
which shows that 
$$\iint_{\Delta}|S_{f}(z)|^2(1-|z|^2)^2dxdy<\infty.$$
 Thus, $f(S^1)$ is a Weil-Petersson curve if and only if the multiplier operator $M_f$ belongs to the Hilbert-Schmidt class. This finishes the proof of Theorem \ref{asy}.

\section{{\bf Final remarks}}
By the experimental work, Kraetzer conjectured in \cite{K} that $B(t)=\frac{t^2}{4}$ when $t\in [-2,2]$. We shall show that $\beta_{f}(-1)\leq \frac{1}{4}$ for certain class of  univalent functions $f$ which admit a quasiconformal extension to $\widehat{\mathbb{C}}$.

We denote by $\mathbf{M}_{\alpha}={\text {Mult}}(\mathcal{H}_{\alpha}(\Delta), \mathcal{H}_{\alpha+4}(\Delta)$ the Banach space of all bounded multipliers form $\mathcal{H}_{\alpha}(\Delta)$ to 
$\mathcal{H}_{\alpha+4}(\Delta)$ supplied with the multiplier norm.  We see that the mapping $a\rightarrow m(az), m\in \mathbf{M}_{\alpha}$ is analytic from $\Delta$ to $\mathbf{M}_{\alpha}$ and continuous from $\overline{\Delta}=\Delta \cup S^{1}$ to $\mathbf{M}_{\alpha}$. On the other hand, the rotation operators $m(z) \rightarrow m(e^{i\theta}z)$ are isometries in $\mathbf{M}_{\alpha}$. Then, by the maximum modulus principle, we know that the dilation operators $m(z)\rightarrow m(rz)$, $r\in (0, 1)$ are contractions in $\mathbf{M}_{\alpha}$(see also \cite{SS}).

Let $f$ be a univalent function in $\Delta$ admitting a quasiconformal extension to $\widehat{\mathbb{C}}$ with $\|\mu_f\|_{\infty}=k$. For $r\in (0,1)$, let $f_{r}(z)={f(rz)}$.
Since
\begin{equation*}
\frac{d^2}{dz^2}\left[(f')^{-1/2}\right]=-\frac{1}{2}S_f(z)(f')^{-1/2},
\end{equation*}
it follows from Lemma \ref{le-1} that for small $\varepsilon>0$, 
 \begin{eqnarray}\label{add-11}
\|(f'_r)^{-1/2}\|_{\alpha}^2& \leq&  \frac{1}{\alpha(\alpha+1)(\alpha+2)(\alpha+3)-\varepsilon}\left\|\frac{d^2}{dz^2}\left[(f'_r)^{-1/2}\right]\right\|_{\alpha+4}^2+\frac{1}{r} C_{13}(f, \varepsilon)\nonumber \\
&= &\frac{1}{4[\alpha(\alpha+1)(\alpha+2)(\alpha+3)-\varepsilon]}\left\|S_{f_r}(z)(f'_r)^{-1/2}\right\|_{\alpha+4}^2+\frac{1}{r} C_{13}(f, \varepsilon).
\end{eqnarray}
Since $S_{f_r}(z)=r^2S_f(rz)$ and since the dilation operators $m(z)\rightarrow m(rz)$, $r\in (0, 1)$ are contractions in $\mathbf{M}_{\alpha}$, we obtain from Theorem \ref{m} that
 \begin{equation}\label{add-12}
\|(f'_r)^{-1/2}\|_{\alpha}^2 \leq \frac{9(\alpha+1)k^2}{(\alpha-1)[\alpha(\alpha+1)(\alpha+2)(\alpha+3)-\varepsilon]}\|(f'_r)^{-1/2}\|_{\alpha}^2+\frac{1}{r} C_{13}(f, \varepsilon)\nonumber. 
\end{equation}
It is not difficult to see that, for any fixed $\alpha>\frac{5}{4}$, when $k\leq \frac{\sqrt{1105}}{48}$ and $\varepsilon$ small enough, we have $$ \frac{9(\alpha+1)k^2}{(\alpha-1)[\alpha(\alpha+1)(\alpha+2)(\alpha+3)-\varepsilon]}< 1.$$
It follows that, for any fixed $\alpha>\frac{5}{4}$,  there are two constants $\mathcal{M}_1>0$ and $\mathcal{R}_1\in (0,1)$ such that $\|(f'_r)^{-1/2}\|_{\alpha}^2\leq \mathcal{M}_1$ when $r>\mathcal{R}_1$. Then, by Fatou's lemma, we have $(f')^{-1/2}\in \mathcal{H}_{\alpha}(\Delta)$ for fixed  $\alpha>\frac{5}{4}$. Hence $\beta_f(-1)\leq \frac{1}{4}$. We have proved that
 \begin{proposition}\label{r-1}
Let $f$ be a univalent function in $\Delta$ admitting a  quasiconformal extension to $\widehat{\mathbb{C}}$.  If $\|\mu_f\|_{\infty} \leq \frac{\sqrt{1105}}{48}\approx 0.69253$, then $\beta_f(-1)\leq \frac{1}{4}$.
\end{proposition}
We also have 
\begin{proposition}\label{r-2}
Let $f$ be a univalent function in $\Delta$ admitting a  quasiconformal extension to $\widehat{\mathbb{C}}$.  Let $h(f)$ be the boundary dilatation of $f$. If $h(f)\leq \sqrt{ \frac{65}{321}}\approx 0.44999$, then $\beta_f(-1)\leq \frac{1}{4}$.
\end{proposition}

\begin{proof}[Proof of Proposition \ref{r-2}]
By Lemma \ref{th-le}, we have, for any $\varepsilon_1 \in (0, 1-h(f))$,  there is a constant $R \in (0,1)$ such that for any $r\in (0,1)$,
\begin{eqnarray*}\iint_{A_R}|S_{f_r}(z)\phi (z)|^2(1-|z|^2)^{\alpha+2}dxdy
\leq \frac{36(h(f)+\varepsilon_1)^2}{(\alpha-1)[1-(h(f)+\varepsilon_1)^2]}\|\phi (z)\|_{\alpha}^2
 \end{eqnarray*}
 holds for any $\alpha>1$ and $\phi \in \mathcal{H}_{\alpha}(\Delta)$.
Thus, for any $\phi \in \mathcal{A}(\Delta)$, we have
\begin{eqnarray}\label{add-111} \|S_{f_r}(z)\phi (rz)\|_{\alpha+4}^2&=&(\alpha+3)\iint_{\Delta}|S_{f_r}(z)\phi (rz)|^2(1-|z|^2)^{\alpha+2}dxdy\nonumber \\
& \leq&\frac{36(\alpha+3)(h(f)+\varepsilon_1)^2}{(\alpha-1)[1-(h(f)+\varepsilon_1)^2]} \|\phi (rz)\|_{\alpha}^2+C_{14}(\phi,  \alpha, h(f), \varepsilon_1).
 \end{eqnarray}
By letting $[f'(rz)]^{-1/2}$ to be instead of $\phi(rz)$ in (\ref{add-111}), we have
 \begin{eqnarray*} \|S_{f_r}(z)(f'_r)^{-1/2}\|_{\alpha+4}^2  \leq\frac{36(\alpha+3)(h(f)+\varepsilon_1)^2}{(\alpha-1)[1-(h(f)+\varepsilon_1)^2]} \|(f'_r)^{-1/2}\|_{\alpha}^2+\frac{1}{r}C_{15}(f,  \alpha, \varepsilon_1).
 \end{eqnarray*}
It follows from  (\ref{add-11}) that, for small number $\varepsilon_2>0$,  
$$
\|(f'_r)^{-1/2}\|_{\alpha}^2 \leq \mathbf{F}(h(f), \varepsilon_1) \mathbf{G}(\alpha, \varepsilon_2)  \|(f'_r)^{-1/2}\|_{\alpha}^2+\frac{1}{r}C_{16}(f,  \alpha, \varepsilon_1, \varepsilon_2),
$$
where 
$$ \mathbf{F}(h(f), \varepsilon_1):=\frac{[h(f)+\varepsilon_1]^2}{1-[(h(f)+\varepsilon_1]^2},$$
$$ \mathbf{G}(\alpha, \varepsilon_2):=\frac{9(\alpha+3)}{(\alpha-1)[\alpha(\alpha+1)(\alpha+2)(\alpha+3)-\varepsilon_2]}.$$

On one hand, since $\mathbf{G}(\alpha, \varepsilon_2)$ is decreasing with respect to $\alpha$ when $\alpha>1$, hence for fixed $\alpha>\frac{5}{4}$ there is a constant $\theta>0$ such that  when $\varepsilon_2$ sufficiently small we have
\begin{equation}\label{ggg} \mathbf{G}(\alpha, \varepsilon_2) <\mathbf{G}(\frac{5}{4}, 0)-\theta=\frac{256}{65}-\theta.\end{equation}

On the other hand, since $x^2(1-x^2)^{-1}$ is increasing in $[0,1)$ and $h(f)+\varepsilon_1<1$, then when
$$h(f)\leq \sqrt{ \frac{65}{321}}\approx 0.44999, \, {\text {i.e.}},\,\frac{[h(f)]^2}{1-[(h(f)]^2} \leq \frac{65}{256},$$
and $\varepsilon_1$ sufficiently small we have  \begin{equation}\label{fff} \mathbf{F}(h(f), \varepsilon_1)<\frac{65}{256}+(\frac{65}{256})^2\theta.\end{equation}
Consequently,  for $\alpha>\frac{5}{4}$ and $\varepsilon_1, \varepsilon_2$ sufficiently small, we see from (\ref{ggg}) and (\ref{fff}) that 
 $$ \mathbf{F}(h(f), \varepsilon_1) \mathbf{G}(\alpha, \varepsilon_2) <1.$$
 It follows that, for fixed $\alpha>\frac{5}{4}$,  $\|(f'_r)^{-1/2}\|_{\alpha}^2\leq \mathcal{M}_2$ for some $\mathcal{M}_2>0$ when $r$ is close enough to $1$. Then, by Fatou's lemma, we see that $(f')^{-1/2}\in \mathcal{H}_{\alpha}(\Delta)$ for any fixed  $\alpha>\frac{5}{4}$. This implies that $\beta_f(-1)\leq \frac{1}{4}$, The proof of Theorem \ref{r-2} is finished.
\end{proof}

Theorem \ref{m-1} and \ref{m-1-1} can be restated  in the language of Teichm\"uller theory. We recall the definition of the universal  Teichm\"uller space and the universal asymptotic Teichm\"uller space. For primary references,  see \cite{GN, EGN-1, EMD, EGN-2}.

Let $M({\Delta}^{*})$ denote the open unit ball of the Banach space $L^{\infty}({\Delta}^{*})$ of essentially bounded measurable functions in ${\Delta}^{*}$. For $\mu \in M({\Delta}^{*})$, let $f_{\mu}$ be the quasiconformal mapping in the extended complex plane $\widehat{\mathbb{C}}$ with complex dilatation equal to $\mu$ in $\Delta^{*}$, equal to $0$ in $\Delta$, normalized $f_{\mu}(0)=0, \, f'_{\mu}(0)=1, \, f_{\mu}(\infty)=\infty$.  We say two elements $\mu$ and $\nu$ in $M(\Delta^{*})$ are equivalent, denoted by $\mu\sim \nu$, if $f_{\mu}|_{\Delta}=f_{\nu}|_{\Delta}$. The equivalence class of $\mu$  is denoted by $[\mu]_{T}$. Then $T=M(\Delta^{*})/\sim$ is one model of the {\em universal Teichm\"uller space}.

The Teichm\"uller distance $d([\mu]_{T}, [\nu]_{T})$ of two points $[\mu]$, $[\nu]$ in $T$ is defined as
\begin{eqnarray}
\lefteqn{d([\mu]_{T}, [\nu]_{T})=\frac{1}{2}\inf  \bigg\{ \log \frac{1+\|(\mu_1-\nu_1)/(1-\overline{\nu_1}\mu_1)\|_{\infty}}{1-\|(\mu_1-\nu_1)/(1-\overline{\nu_1}\mu_1)\|_{\infty}},}
\nonumber \\
&&\quad\quad\qquad\qquad\qquad\qquad\qquad\qquad [\mu_1]_{T}=[\mu]_{T}, [\nu_1]_{T}=[\nu]_{T}  \bigg\}. \nonumber
\end{eqnarray}
In particular, the distance between $[\mu]_T$ and the basepoint $[0]_T$ is
$$d([\mu]_T, [0]_T)=\frac{1}{2}\log \frac{1+k_0([\mu]_T)}{1-k_0([\mu]_T)}, \,k_0([\mu]_T)=\inf\{\|\nu\|_{\infty},\, \nu\sim \mu\}.$$

We say $\mu$ and $\nu$ in $M({\Delta}^{*})$ are asymptotically equivalent if there exists some $\tilde{\nu}$ such that $\tilde{\nu}$ and $\nu$ are equivalent and $\tilde{\nu}(z)-\mu(z) \rightarrow 0$ as $|z|\rightarrow 1^{+}$. The asymptotic equivalence
of $\mu$ will be denoted by $[\mu]_{AT}$. The {\em universal asymptotic Teichm\"uller space} $AT$ is the set of
all the asymptotic equivalence classes $[\mu]_{AT}$ of elements $\mu$ in $M({\Delta}^{*})$.

The Teichm\"uller distance $d([\mu]_{AT}, [\nu]_{AT})$ of two points $[\mu]_{AT}$, $[\nu]_{AT}$ in $AT$ is defined as
\begin{eqnarray}
\lefteqn{d([\mu]_{AT}, [\nu]_{AT})=\frac{1}{2}\inf  \bigg\{ \log \frac{1+\|(\mu_1-\nu_1)/(1-\overline{\nu_1}\mu_1)\|_{\infty}}{1-\|(\mu_1-\nu_1)/(1-\overline{\nu_1}\mu_1)\|_{\infty}},}
\nonumber \\
&&\quad\quad\qquad\qquad\qquad\qquad\qquad\qquad [\mu_1]_{AT}=[\mu]_{AT}, [\nu_1]_{AT}=[\nu]_{AT}  \bigg\}. \nonumber
\end{eqnarray}
In particular, the distance between $[\mu]_{AT}$ and the basepoint $[0]_{AT}$ is
$$d([\mu]_{AT}, [0]_{AT})=\frac{1}{2}\log \frac{1+h_0([\mu]_{AT})}{1-h_0([\mu]_{AT})},\,h_0([\mu]_{AT})=\inf\{h([\nu]_T),\,[\nu]_{AT}=[\mu]_{AT}\}.$$
Here,
$$h([v]_T):=\inf\{h^{*}(\mu): \, {\mu\sim \nu}\},$$
and $$ h^{*}(\mu)=\inf\{\|\mu|_{\Delta^{*}\setminus E}\|_{\infty}: \, E {\text { is a compact set in}}\,\,\Delta^{*}\}.$$
\begin{remark}
 It is known that
$h_0([\mu]_{AT})=h([\mu]_T)$ and it is easy to see that $h([\mu]_{T})=\inf \{h(f_{\nu}),\,\, {\nu\sim \mu}\}$, here $h(f_{\nu})$  is defined as in (\ref{boun}). \end{remark}

Then, we can restate Theorem \ref{m-1} and Theorem \ref{m-1-1} as
\begin{theorem}
Let $\mu\in M({\Delta}^{*})$. Let $f_{\mu}$ be a quasiconformal mapping in the extended complex plane $\widehat{\mathbb{C}}$ with complex dilatation equal to $\mu$ in $\Delta^{*}$, equal to $0$ in $\Delta$, normalized $f_{\mu}(0)=0, \, f'_{\mu}(0)=1, \, f_{\mu}(\infty)=\infty$. If
$$d([\mu]_T, [0]_T)\leq \frac{1}{2}\log \frac{1+\sqrt{5/8}}{1-\sqrt{5/8}},\,\,{\text {or}}\,\, d([\mu]_{AT}, [0]_{AT})\leq \frac{1}{2}\log \frac{1+\sqrt{3/11}}{1-\sqrt{3/11}},$$
then $\beta_{f_{\mu}}(-2)\leq 1$ and the Brennan conjecture is satisfied for the domain $f_{\mu}(\Delta).$
\end{theorem}

\section{{\bf {Acknowledgements}}}
The author would like to thank the referee for careful reading of the paper and  invaluable comments and suggestions.  The author also thanks Prof. Shen Yuliang for pointing out several errors in the earlier version of  this paper.


\begin{thebibliography}{99}
\bibitem{AZ}Astala K., Zinsmeister M.,  {\em Teichm\"uller spaces and BMOA}, Math. Ann., 289(1991), pp. 613-625.

\bibitem{BVZ}Bara\'nski A., Volberg A., and Zdunik A., {\em Brennan’s conjecture and the Mandelbrot Set}, Int. Math. Res. Not., No. 12, 1998.


\bibitem{B-1}Bishop C., {\em Function theoretic characterizations of Weil-Petersson curves}, Rev. Mat. Iberoam., 38 (2022), no. 7, pp. 2355-2384.

\bibitem{B-2}Bishop C., {\em Weil-Petersson curves, beta-numbers, and minimal surfaces}, preprint.

\bibitem{Carl}Carleson L., Jones P., {\em On coefficient problems for univalent functions and conformal dimension},  Duke Math. J.,  66 (1992), no. 2, pp. 169-206.

\bibitem{CM}Carleson L., Makarov N.,  {\em Some results connected with Brennan's conjecture},  Ark. Mat., 32 (1994), no. 1, pp. 33-62.


\bibitem{C}Cui G.,  {\em Integrably asymptotic affine homeomorphisms of the circle and Teichm\"uller spaces},  Science in China(Series A),  43(2017), pp. 267-279.

\bibitem{EGN-1}Earle C., Gardiner F.,  and Lakic N., {\em Asymptotic Teichm\"uler space I: The complex structure},  Contemp. Math.,  256 (2000), pp. 17-38.

\bibitem{EMD}Earle C., Markovic V.,  and Saric D., {\em Barycentric extension and the Bers embedding for asymptotic Teichm\"uller space},  Contemp. Math., 311(2002), pp. 87-105.

\bibitem{EGN-2}Earle C., Gardiner F., and Lakic N., {\em Asymptotic Teichm\"uller space II: The metric structure}, Contemp. Math., 355 (2004), pp. 187-219.

\bibitem{GN}Gardiner F., Lakic N., {\em Quasiconformal Teichm\"uller theory}, Mathematical Surveys and Monographs, 76,  American Mathematical Society, Providence, RI, 2000.

\bibitem{GS}Gardiner F., Sullivan D., {\em Symmetric structures on a closed curve}, Amer. J. Math., 114(1992), pp. 683-736.


\bibitem{GM}Garnett J., Marshall D., {\em Harmonic Measure}, New Math. Monogr., vol. 2, Cambridge Univ. Press, Cambridge, 2005.

\bibitem{HSo}Hedenmalm H.,  Sola A., {\em Spectral notions for conformal maps: a survey}, Comput. Methods Funct. Theory,  8 (2008), no. 1-2, pp. 447-474.

\bibitem{H}Hedenmalm H., {\em Bloch functions and asymptotic tail variance},  Adv. Math.,  313(2017), pp. 947-990

\bibitem{HS-1}Hedenmalm H., Shimorin S., {\em Weighted Bergman spaces and the integral means spectrum of conformal mappings},  Duke Math. J. 127 (2005), no. 2, pp. 341-393.

\bibitem{HS}Hu Y., Shen Y., {\em On quasisymmetric homeomorphisms}, Israel J. Math., (191) 2012, pp. 209-226.

\bibitem{K} Kraetzer P., {\em Experimental bounds for the universal integral means spectrum of conformal maps}, Complex Variables Theory Appl., (31) 1996, pp. 305-309.

\bibitem{L}Lehto O., {\em Univalent functions and Teichm\"uller spaces}, Springer-Verlag, 1987.

\bibitem{LV}Lehto O., Virtanen K., {\em Quasiconformal mappings in the plane}, Second edition, Springer-Verlag, New York-Heidelberg, 1973.

\bibitem{M}Makarov N.,  {\em Fine structure of harmonic measure},  St. Petersburg Math. J.,  10 (1999), no. 2, pp. 217-268.

\bibitem{Po}Pommerenke C., {\em Univalent functions}, Vandenhoeck and Ruprecht, G\"ottingen, 1975.

\bibitem{Po-1}Pommerenke C., \textit{On univalent functions, Bloch functions and VMOA}, Math. Ann., 236(1978), pp. 199-208.

\bibitem{Sh-1}Shen Y., {\em On Grunsky operator}. Sci. China Ser.  A,  50 (2007), no. 12, pp. 1805-1817.

\bibitem{Sh-2}Shen Y., {\em Weil-Petersson Teichmüller space},  Amer. J. Math., 140 (2018), no. 4,  pp. 1041-1074.


\bibitem{SS}Shimorin S., {\em A multiplier estimate of the Schwarzian derivative of univalent functions}, Int. Math. Res. Not., No. 30, 2003.

\bibitem{W}Wang Y., {\em Equivalent descriptions of the Loewner energy}, Invent. Math., 218(2019), no. 2, pp. 573-621.

\bibitem{Zh}Zhu K.,  {\em  Operator theory in function spaces},  Monographs and Textbooks in Pure and Applied Mathematics, 139, Marcel Dekker, Inc., New York, 1990.

\end{thebibliography}
\end{document}